\documentclass[a4paper]{article}
\usepackage[total={6in, 8in}]{geometry}
\usepackage[utf8]{inputenc}
\usepackage{amsmath, xpatch}
\usepackage{mathtools}
\usepackage{subcaption}
\usepackage{graphicx}
\usepackage{amsthm,amssymb}
\usepackage{bbm}
\usepackage{booktabs, dcolumn}
\usepackage{array,multirow}
\usepackage{mathdots}
\usepackage{placeins}
\usepackage{adjustbox}

\usepackage{amsfonts}
\usepackage{color}
\usepackage{authblk}
\usepackage[colorlinks]{hyperref}
\usepackage[colorinlistoftodos]{todonotes}
\paperwidth=\dimexpr \paperwidth + 6cm\relax
\oddsidemargin=\dimexpr\oddsidemargin + 3cm\relax
\evensidemargin=\dimexpr\evensidemargin + 3cm\relax
\marginparwidth=\dimexpr \marginparwidth + 3cm\relax

\allowdisplaybreaks
\newcommand{\bm}{\mathbf{m}}
\newcommand{\bu}{\mathbf{u}}

\newcommand{\G}{\mathcal{G}}
\newcommand{\V}{\mathcal{V}}
\newcommand{\E}{\mathcal{E}}

\title{Measure-theoretic bounds on the spectral radius of graphs from walks}
\author[1]{Francisco Barreras \footnote{Corresponding author. Address: University of Pennsylvania, Philadelphia, PA 19104, USA}}
\author[2]{Mikhail Hayhoe}
\author[2]{Hamed Hassani}
\author[2]{Victor M. Preciado}
\affil[1]{Department of Mathematics, University of Pennsylvania, Philadelphia, PA 19104, USA}
\affil[2]{Department of Electrical \& Systems Engineering, University of Pennsylvania, Philadelphia, PA 19104, USA}
\date{October 2019}

\begin{document}
\theoremstyle{definition}
\makeatletter
\xpatchcmd{\@thm}{.}{}{}{}
\makeatother

\makeatletter
\newcommand\blfootnote[1]{%
  \begingroup
  \renewcommand\thefootnote{}\footnote{#1}%
  \addtocounter{footnote}{-1}%
  \endgroup}
\makeatother

\newtheorem{theorem}{Theorem}
\newtheorem{definition}[theorem]{Definition}
\newtheorem{lemma}[theorem]{Lemma}
\newtheorem{corollary}[theorem]{Corollary}
\maketitle
\blfootnote{\textit{Email addresses}: fbarrer@sas.upenn.edu, mhayhoe@seas.upenn.edu, hassani@seas.upenn.edu, \mbox{preciado@seas.upenn.edu}} 
\begin{abstract}
Let $\G$ be an undirected graph with adjacency matrix $A$
and spectral radius $\rho$. Let $w_k, \phi_k$ and
$\phi_k^{(i)}$ be, respectively, the number walks of
length $k$, closed walks of length $k$ and closed walks
starting and ending at vertex $i$ after $k$ steps. In this
paper, we propose a measure-theoretic framework which
allows us to relate walks in a graph with its spectral 
properties. In particular, we show that $w_k, 
\phi_k$ and $\phi_k^{(i)}$ can be interpreted as the 
moments of three different measures, all of them supported
on the spectrum of $A$.  Building on this interpretation, 
we leverage results from the classical moment problem to 
formulate a hierarchy of new lower and upper bounds on 
$\rho$, as well as provide alternative proofs to several 
well-known bounds in the literature. \\

\end{abstract}
\section{Introduction}

Given an undirected graph $\G = (\V, \E)$ with vertex set 
$\V = \{1, \dots, n \}$ and edge set $\mathcal{E}\subseteq
\V \times \V$, we define a \textit{walk of length $k$}, or
\textit{$k$-walk}, as a sequence of vertices $(i_0, i_1,
\dots, i_{k})$ such that $(i_s, i_{s+1}) \in \E$ for $s \in \{0, \dots , k-1\}$. A walk of length $k$ is called \textit{closed} if $i_0 = i_{k}$; furthermore, we will refer to them as \textit{closed walk from vertex} $i_0$ when we need to distinguish them from the set of all closed $k$-walks. We denote the number of walks, closed walks, and closed walks from vertex $i$ of length $k$ by $w_k$, $\phi_k$, and $\phi_k(i)$, respectively. Denote by $A$ the adjacency matrix of  graph $\G$ and its eigenvalues by $\lambda_1 \ge \lambda_2 \ge \cdots \ge \lambda_n$ . The set containing these eigenvalues will be referred to as the \textit{spectrum} of $\G$. From Perron-Frobenius' Theorem~\cite{horn2012matrix}, we have that the spectral radius of $A$, defined by $\rho:= \max\limits_{\scriptscriptstyle 1\le i\le n} \vert \lambda_i\vert$, is equal to $\lambda_1$. \\

In the literature, we find several lower bounds on $\rho$ formulated in terms of walks in the graph. Bounds in terms of closed walks are rarer (see, e.g, \cite{preciado2013moment}). Many of these bounds come from dexterous applications of the Rayleigh principle \cite{horn2012matrix} or the Cauchy-Schwarz inequality. For example, making use of these tools,  Collatz and Sinogowitz~\cite{collatz1957spektren}, Hoffmeister~\cite{hofmeister1988spectral}, Yu et. al.~\cite{yu2004spectral}, and  Hong and Zhang~\cite{hong2005sharp} derived, respectively, the following lower bounds:
\begin{equation}
        \rho \ge \dfrac{w_1}{w_0},
\qquad \rho \ge \sqrt{\dfrac{w_2}{w_0}},
\qquad \rho \ge \sqrt{\dfrac{w_4}{w_2}},
\qquad \rho \ge \sqrt{\dfrac{w_6}{w_4}}.  \label{eq:first-bounds}
\end{equation}

 Nikiforov\footnote{Nikiforov's notation in \cite{nikiforov2006walks} indexes $w_k$ in terms of the number of nodes visited by the walks instead of the number of steps, as used in our manuscript.} \cite{nikiforov2006walks} generalized these results by expressing the number of walks $w_k$ in terms of the eigenvalues, to obtain bounds of the form
\begin{align}
    \rho^r \ge \dfrac{w_{2s+r}}{w_{2s}}, \label{niki-bound}
\end{align}
for $s,r \in \mathbb{N}_0$. Cioab\u{a} and Gregory  \cite{cioabua2007large} provide the following improvement to the first bound in \eqref{eq:first-bounds}:
\begin{align}
    \rho \ge \dfrac{w_1}{w_0} + \dfrac{1}{w_0(\Delta + 2)},
\end{align}

\noindent where $\Delta$ is the maximum of the vertex degrees in $\G$. Nikiforov \cite{nikiforov2007bounds} showed that
\begin{align}
    \rho > \frac{w_1}{w_0} + \frac{1}{2w_0 + w_1}.
\end{align}

\noindent Favaron et. al \cite{favaron1993some} used the fact that there is a $K_{1,{\scriptscriptstyle \Delta}}$ subgraph in $\G$ to obtain:
\begin{align}
    \rho \ge \sqrt{\Delta}. \label{favaron}
\end{align}

There is a number of upper bounds on $\rho$ in terms of graph invariants like the domination number \cite{stevanovic2008spectral}, chromatic number \cite{cvetkovic1972chromatic, edwards1983lower}, and clique number \cite{edwards1983lower}. Nikiforov \cite{nikiforov2006walks} provides a whole hierarchy of bounds in terms of the clique number $\omega(\G)$, which for $k \in \mathbb{N}_0$ are given by 
\begin{align}
    \rho^{k+1} \le \left(1 - \dfrac{1}{\omega(\G)}\right) w_{k}. \label{nikiforov-bound}
\end{align}
\noindent We also find in the literature several bounds in terms of the \textit{fundamental weight} of $\G$, defined as $\sum_{j=1}^n u_{1j}$, where $u_{1j}$ is the $j$-th entry of the \textit{leading eigenvector}\footnote{The leading eigenvector of $A$ is the eigenvector associated with the largest eigenvalue $\lambda_1$. We assume eigenvectors to be normalized to be of unit Euclidean norm.} of $A$ denoted by $\mathbf{u}_1$. For example, Wilf \cite{wilf1986spectral} proved the following upper bound:
\begin{align}
\rho \leq \frac{\omega(\mathcal{G})-1}{\omega(\mathcal{G})}\left(\sum_{j=1}^{n} u_{1j}\right)^{2}. \label{wilf}   
\end{align}
 Cioaba and Gregory \cite{cioaba2007principal} showed that, for $k \in \mathbb{N}_0$,
\begin{align}
    \rho^k \le \sqrt{w_{2k}}\max_{1\le j \le n} u_{1j}. \label{ineq:vanmieghem} 
\end{align}
Moreover, Van Mieghem \cite{van2014graph} proved the bound
\begin{align}
    \rho^k \le \dfrac{w_k}{\sum_{i=1}^n u_{1i}}\max_{1\le j \le n}u_{1j}. 
\end{align}

In this paper we provide upper and lower bounds on $\rho$ by interpreting the sequences $\{w_k\}_{k=0}^\infty, \{\phi_k\}_{k=0}^\infty$ and $\{\phi_k^{(i)}\}_{k=0}^\infty$ as moments of three measures supported on the spectrum of $\G$. Building on this interpretation, we will use classical results from probability theory relating the moments of a measure with its support. Following this approach, we will derive a hierarchy of new bounds on the spectral radius, as well as provide alternative proofs to several existing bounds in the literature. The rest of the paper is organized as follows. Section~\ref{sec:background} outlines the tools we will use to analyze walks on graphs using measures and moment sequences. Section~\ref{sec:lower} presents multiple lower bounds on the spectral radius derived from the moment problem, while Section~\ref{sec:upper} introduces several upper bounds.
%%%%%%%%%%%%%%%%%%%%%%%%%%%%%%%%%%%%%%%%%%%%%%%%%%%%%%
%BACKGROUND AND PRELIMINARIES
%%%%%%%%%%%%%%%%%%%%%%%%%%%%%%%%%%%%%%%%%%%%%%%%%%%%%%
\section{Background and Preliminaries} \label{sec:background}
Throughout this paper, we use standard graph theory notation, as in \cite{west1996introduction}. We will use upper-case letters for matrices, calligraphic upper-case letters for sets, and bold lower-case letters for vectors. For a vector $\mathbf{v}$ or a matrix $M$, we denote by $\mathbf{v}^\intercal$ and $M^\intercal$ their respective transposes. The $(i,j)$-th entry of a matrix $M$ is denoted by $M_{ij}$. For a $n\times n$ matrix $M$ and a set $\mathcal{J} \subseteq \{1, \dots, n\}$, the matrix $M_{\mathcal{J}}$ is defined to be the submatrix of $M$ where columns and rows with indices not in $\mathcal{J}$ have been removed; $M_{\mathcal{J}}$ is also called a \textit{principal submatrix} of $M$ and, if $\mathcal{J} = \{1, \dots, k\}$, $M_{\mathcal{J}}$ is called a \textit{leading principal submatrix}. Finally, we say that a symmetric matrix $M \in \mathbb{R}^{n \times n}$ is positive semidefinite (resp. positive definite) if for every non-zero vector $\mathbf{v} \in \mathbb{R}^n$ we have $\mathbf{v}^\intercal M \mathbf{v} \geq  0$ (resp. $\mathbf{v}^\intercal M \mathbf{v} > 0 )$ and we denote this as $M \succeq 0$ (resp. $M \succ 0$).

\subsection{Spectral measures and walks}
We can relate walks and closed walks on a graph $\G$ to its spectrum using measures, as we describe in detail below. We begin by stating the following well-known result from algebraic graph theory \cite{biggs1993algebraic}:
\begin{lemma} \label{lema:powers}
    For any integer $k$, the $(i,j)$-th entry of the matrix $A^k$ is equal to the number of $k$-walks from vertex $i$ to vertex $j$ on $\G$.
\end{lemma}

\noindent Since $\G$ is undirected, $A$ is symmetric and admits an
orthonormal diagonalization. In particular, let $\{\bu_1, \bu_2,
\ldots, \bu_n\}$ be a complete set of orthonormal eigenvectors of $A$. Hence, we have that $    A^k = U
\operatorname{diag}\left(\lambda_{1}^k, \ldots, \lambda_{n}^k\right)
U^{\intercal},$
\noindent for every $k \ge 0$, where $U := [\bu_1 | \bu_2 | \ldots |
\bu_n]$. We denote the $i$-th entry of the $l$-th eigenvector by $u_{il}$. From this factorization, we can obtain identities which will
be used in the following sections.
\begin{lemma} \label{lema:moms}
Define $c_l^{(i)} := u_{il}^2$ and $c_l := \left(\sum_{i=1}^n u_{il}\right)^2$. Then, for every $k \ge 0$, we have
\begin{align*}
    \phi_k   = \sum_{l=1}^n \lambda_l^k, \qquad
    \phi_k(i) = \sum_{l=1}^n  c_l^{(i)} \lambda_l^k, \qquad 
    w_k  = \sum_{l=1}^n c_l  \lambda_l^k. 
\end{align*}
\end{lemma}
\begin{proof}
 Using Lemma~\ref{lema:powers} we have that $\phi_k   = \sum_{i=1}^n (A^k)_{ii}$, $\phi_k(i) = (A^k)_{ii}$, and  $w_k = \sum_{i, j} (A^k)_{ij}$. Furthermore, we have that $\left(A^{k}\right)_{i j} =\sum_{l=1}^{n} u_{i l} \lambda_{l}^{k} u_{j l},$ directly from the diagonalization of $A$. Combining these results, and the fact that $\sum_{i=1}^n c_l^{(i)} = 1$, the result follows.
\end{proof}
Next, we introduce three atomic measures supported on the spectrum of $A$.

\begin{definition}[Spectral measures] \label{def:spectral-measures}
 Let $\delta(\cdot)$ be the Dirac delta measure. For a simple graph $\G$ with eigenvalues $\lambda_1 \ge \lambda_2, \dots \ge \lambda_n$, define the \textit{closed-walks measure} as
 \begin{align*}
     \mu_{\G}(x) \coloneqq \sum_{l=1}^{n} \delta\left(x-\lambda_{l}\right).
 \end{align*}
We also define the \textit{closed-walks measure for vertex $i$} as
\begin{align*}
    \mu_{\G}^{(i)}(x) \coloneqq \sum_{l=1}^{n} c_l^{(i)} \delta\left(x-\lambda_{l}\right),
\end{align*}
\noindent and the \textit{walks measure} as
 \begin{align*}
     \nu_{\G}(x) \coloneqq \sum_{l=1}^{n} c_l \delta\left(x-\lambda_{l}\right).
 \end{align*}
\end{definition}

\begin{lemma} \label{lemma:moms-walks} 
For a real measure $\zeta (x)$, define its \textit{$k$-th moment} as $m_k\left(\zeta\right) = \int_{\mathbb{R}} x^k \mathrm{d} \zeta(x)$. Then, the measures in Definition \eqref{def:spectral-measures} satisfy
\begin{align*} 
    m_k\left(\mu_{\G}\right) = \phi_k, \qquad
    m_k (\mu_{\G}^{(i)}) = \phi_k^{(i)}, \qquad
    m_k\left(\nu_{\G}\right) = w_k.  
\end{align*}
\end{lemma}
\begin{proof}
 For the case of $\mu_{\G}$, we evaluate $k$-th moment, as follows:
 \begin{align*}
     m_k\left(\mu_{\G}\right) = \int_{\mathbb{R}} x^k \mathrm{d} \mu_{\G}(x) = \int_{\mathbb{R}}x^k \sum_{l=1}^n \delta(x - \lambda_l) \mathrm{d}x =\sum_{l=1}^n \lambda_l^k = \phi_k.
 \end{align*}
 The other two cases have analogous proofs.
\end{proof}

\subsection{The moment problem}

In order to derive bounds on the spectral radius $\rho$, we will make use of results from the \textit{moment problem} \cite{schmudgen2017moment}. This problem is concerned with finding necessary and sufficient conditions for a sequence of real numbers to be the \textit{moment sequence} of a measure supported on a set $\mathcal{K} \subseteq \mathbb{R}$. This is formalized below.

\begin{definition}[$\mathcal{K}$-moment sequence] \label{def:moment-seq}
 The infinite sequence of real numbers $\mathbf{m}=\left(m_{0}, m_{1}, m_{2},\ldots\right)$ is called a \emph{$\mathcal{K}$-moment sequence} if there exists a Borel measure $\zeta$ supported on $\mathcal{K} \subseteq \mathbb{R}$ such that
\begin{align*}
    m_k = \int_{\mathcal{K}} x^k d\zeta(x), \hspace{12pt} \text{ for all } k \in \mathbb{N}_0 .
\end{align*}
\end{definition}

The following result, known as \textit{Hamburger's theorem} \cite{schmudgen2017moment}, will be used in Sections \ref{sec:lower} and \ref{sec:upper}.

\begin{theorem}[Hamburger's Theorem~\cite{schmudgen2017moment}]\label{thm:hamburger} 
Let $\mathbf{m} = \left(m_{0},m_{1},m_{2},\ldots\right)$ be an infinite sequence of real numbers. For $n \in \mathbb{N}_0$, define the \textit{Hankel matrix of moments} as
\begin{align}
    H_{n}(\mathbf{m}) \coloneqq \left[\begin{array}{cccc}
    {m_{0}} & {m_{1}} & {\dots} & {m_{n}} \\
    {m_{1}} & {m_{2}} & {\dots} & {m_{n+1}} \\
    {\vdots} & {\vdots} & {\ddots} & {\vdots} \\
    {m_{n}} & {m_{n+1}} & {\dots} & {m_{2 n}}
    \end{array}\right] \in \mathbb{R}^{(n+1)\times (n+1)}. \label{def:hankel-moms}
\end{align}
The sequence $\mathbf{m}$ is a $\mathbb{R}$-moment sequence, if and only if, for every $n \in \mathbb{N}_0$, $H_n(\mathbf{m}) \succeq 0$.
\end{theorem}

The characterizations of moment sequences supported on intervals of the form $(-\infty, u]$ and $[-u, u]$ are known as the Stieltjes and Hausdorff moment problems, respectively. A proof for the following theorem, known as \textit{Stieltjes' theorem}, can be found in \cite{schmudgen2017moment} for the case where $u=0$, and it can be easily adapted to any $ u \in \mathbb{R}$ through a simple change of variables.
\begin{theorem}[Stieltjes' theorem] \label{thm:stieltjes}
Let $\mathbf{m} = \left(m_{0},m_{1},m_{2},\ldots\right)$ be an infinite sequence of real numbers. For $n \in \mathbb{N}_0$, define the \textit{shifted Hankel matrix of moments} $S_n(\mathbf{m})$ as
\begin{align}
    S_{n}(\mathbf{m}) \coloneqq \left[\begin{array}{cccc}
    {m_{1}} & {m_{2}} & {\dots} & {m_{n+1}} \\
    {m_{2}} & {m_{3}} & {\dots} & {m_{n+2}} \\
    {\vdots} & {\vdots} & {\ddots} & {\vdots} \\
    {m_{n+1}} & {m_{n+2}} & {\dots} & {m_{2 n + 1}}
    \end{array}\right]\in \mathbb{R}^{(n+1)\times (n+1)}.
\end{align}
The sequence $\mathbf{m}$ is a $(-\infty, u]$-moment sequence, if and only if, for every $n \in \mathbb{N}_0$,
\begin{align}
    H_n(\mathbf{m}) \succeq 0, \quad \text{ and } \quad u H_n(\mathbf{m}) - S_n(\mathbf{m}) \succeq 0. \label{mom-minus}
\end{align}
Similarly, the sequence $\mathbf{m}$ is a $[-u, \infty)$-moment sequence, if and only if, for every $n \in \mathbb{N}_0$,
\begin{align}
    H_n(\mathbf{m}) \succeq 0, \quad \text{ and } \quad  u H_n(\mathbf{m}) + S_n(\mathbf{m}) \succeq 0 \label{mom-plus}.
\end{align}
\end{theorem}

The positive (semi)definiteness of a symmetric matrix can be certified using \textit{Sylvester's criterion}. 
\begin{theorem}[Sylvester's criterion \cite{meyer2000matrix}]
A matrix $M$ is positive semidefinite, if and only if, the determinant of every principal submatrix is non-negative. Moreover, $M$ is positive definite, if and only if, the determinant of every leading principal submatrix $R$ is positive.
\end{theorem}
%%%%%%%%%%%%%%%%%%%%%%%%%%%%%%%%%%%%%%%%%%%%%%%%%%%%%%%%%%%
%LOWER BOUNDS ON THE SPECTRAL RADIUS
%%%%%%%%%%%%%%%%%%%%%%%%%%%%%%%%%%%%%%%%%%%%%%%%%%%%%%%%%%%
\section{Lower Bounds on the Spectral Radius} \label{sec:lower}

The supports of the spectral measures in Definition \eqref{def:spectral-measures} are contained in the interval 
$[-\rho, \rho]$ and their moments can be written in terms of walks in $\G$. Since the moments of a measure impose constraints on its support, the number of walks in $\G$ imposes constraints on $\rho$, as stated below.

\begin{lemma} \label{lem:sdp}
For a graph $\G$, let $\bm$ be the sequence of moments of any measure supported on the spectrum of $\G$. Then, for any finite set $\mathcal{J} \subset \mathbb{N} _0$,
\begin{align}
    \rho H_{\mathcal{J}}(\mathbf{m}) - S_{\mathcal{J}}(\mathbf{m}) &\succeq 0, \label{mom-minus-J}\\ 
     \rho H_{\mathcal{J}}(\mathbf{m}) + S_{\mathcal{J}}(\mathbf{m}) &\succeq 0  \label{mom-plus-J} ,
\end{align}
\noindent where $H_{\mathcal{J}}(\bm)$ and $S_{\mathcal{J}}(\bm)$ are submatrices of $H_n(\bm)$ and $S_n(\bm)$, defined in \eqref{thm:hamburger} and \eqref{thm:stieltjes}, respectively.
\end{lemma}
\begin{proof}
Since $\mathbf{m}$ corresponds to the sequence of moments of a measure whose support is contained in $[-\rho, \rho]$, it follows that $\rho$ must satisfy the necessary conditions $\eqref{mom-minus}$ and $\eqref{mom-plus}$. Furthermore, every leading principal submatrix of a positive semidefinite matrix is also positive semidefinite by Sylvester's criterion; hence, the matrix inequalities \eqref{mom-minus-J} and \eqref{mom-plus-J} follow.
\end{proof}

As stated in Lemma~2.3, the moments of all the three measures defined in Definition \ref{def:spectral-measures} can be written in terms of walks in the graph. Since the supports of these three measures are equal to the eigenvalue spectrum of $\G$, we can apply Lemma~\ref{lem:sdp} to the moment sequences obtained by counting different types of walks in the graph. Using the above Lemma, we can use a truncated sequence of moments to find a lower bound on $\rho$ by solving a \textit{semidefinite program} \cite{boyd2004convex}, as stated below:

\begin{theorem} \label{thm:numerical}
The solution to the following semidefinite program is a lower bound on the spectral radius of $\G$
\begin{align*}
\begin{array}{cl}
{\displaystyle \min_{u}} & u \\
{\text { s.t. }} & {uH_{n}(\mathbf{m}) - S_{n}(\mathbf{m}) \succeq 0,} \\
{} & {uH_{n}(\mathbf{m}) + S_{n}(\mathbf{m}) \succeq 0,}
\end{array}
\end{align*}
where $\mathbf{m} = (m_0, m_1, \dots, m_{2n + 1})$ is a truncated sequence of moments of any measure supported on the eigenvalue spectrum of $\G$.
\end{theorem}

The above Theorem~can be used to compute numerical bounds on the spectral radius by setting the moments to be one of $\phi_k, \phi_k^{(i)}$, or $w_k$. Moreover, we can use Lemma~\ref{lem:sdp} to obtain closed-form bounds on $\rho$ involving a small number of moments for which the semidefinite program in Theorem~\ref{thm:numerical} can be solved analytically. The following corollary analyzes the case where $|\mathcal{J}| = 1$.
\begin{corollary}  \label{basic-niki-bound}
For an undirected graph $\G$, let $\bm$ be the sequence of moments of a measure supported on the spectrum of $\G$. Then, for every $k \in \mathbb{N}_0$ and even $q$,
\begin{align}
    \rho^k \ge \dfrac{m_{2s+k}}{m_{2s}}. \label{eq:basic-niki-bound}
\end{align}
\end{corollary}
\begin{proof}
Let $\zeta(x)$ be an atomic measure supported on $\{\lambda_1, \lambda_2, \dots, \lambda_n \}$, defined as $\zeta(x) \coloneqq \sum_{i=1}^n z_i \delta(x - \lambda_i)$ and let $\{m_0, m_1, \dots \}$ be its moment sequence. For every $k \in \mathbb{N}_0$ and even $q$, we construct the following measure based on $\zeta(x)$:
\begin{align}
\zeta_{q,k}(x) \coloneqq \sum_{i=1}^n z_i \lambda_i^q \delta(x - \lambda_i^k).    \label{eq:new-measure}
\end{align}

We see that $\zeta_{q,k}(x)$ is supported on $\{\lambda_1^k, \lambda_2^k, \dots, \lambda_n^k \}$, and its moments are given by the sequence $\{m_q, m_{q+k}, m_{q+2k}, \dots \}$, for $k \in \mathbb{N}_0$. We note that, for even $q$, the support of the measure $\zeta_{q,k}(x)$ is contained in $[-\rho^k, \rho^k]$; thus, setting $J = \{1\}$, we use Lemma~\ref{lem:sdp} to obtain $\rho^k m_{q} - m_{q+k} \ge 0$, which implies \eqref{eq:basic-niki-bound}.
\end{proof}

If we set $\mathbf{m} = \{w_s\}_{s=0}^{\infty}$, this corollary gives an alternative proof for the lower bounds in \eqref{niki-bound}, proven by Nikiforov \cite{nikiforov2006walks}. It also generalizes these results to closed walks by using $\phi_k$ or $\phi_k^{(i)}$ as the sequence of moments. 

Another interesting result comes from applying Lemma~\ref{lem:sdp} to the case where $|\mathcal{J}| = 2$. Corollary \ref{cor:largest-root} below provides a new lower bound in terms of the largest root of a quadratic polynomial. Its proof relies on the following lemma.

\begin{lemma} \label{lemma:dets}
Let $\bm$ be the sequence of moments of a measure supported on the spectrum of $\G$. For $s, k \ge 0$, define the following matrices:
\begin{align}
    H^{(2s,k)}\coloneqq \left[\begin{array}{ll}{m_{2s}} & {m_{2s+k}} \\ {m_{2s+k}} & {m_{2s+2k}}\end{array}\right], \quad \text{and} \quad S^{(2s,k)}\coloneqq\left[\begin{array}{ll}{m_{2s+k}} & {m_{2s+2k}} \\ {m_{2s+2k}} & {m_{2 s + 3k}}\end{array}\right] . \label{eq:HS-matrices}
\end{align}
Whenever $\det\left(H^{(2s,k)}\right) \neq 0$, we have
\begin{align}
    \rho^{2k} \ge \dfrac{  \det(S^{(2s,k)})}{\det(H^{(2s,k)})}. \label{H-S}
\end{align}
\end{lemma}
\begin{proof}
Let $\zeta_{2s,k}(x)$ be the measure defined in \eqref{eq:new-measure}. The support of this measure is supported on $[-\rho^k, \rho^k]$ and has moment sequence $(m_{2s}, m_{2s+k}, m_{2s+2k}, \dots )$. Applying Lemma~\ref{lem:sdp} with $\mathcal{J} = \{1, 2\}$ we obtain
\begin{align}
    \rho^{k} H^{(2s,k)} \pm S^{(2s,k)} \succeq 0 \implies \rho^k \ge \dfrac{ \mathbf{x}^{\intercal}S^{(2s,k)} \mathbf{x} }{\mathbf{x}^{\intercal} H^{(2s,k)} \mathbf{x}}, \label{H-S posdef} 
\end{align}
\noindent for every non-zero $\mathbf{x} \in \mathbb{R}^2$. From Theorem~\ref{thm:hamburger}, we know that $H^{(2s, k)} \succeq 0$; hence, its eigenvalues $\xi_1$ and $\xi_2$ satisfy $\xi_1 \ge \xi_2 \ge 0$. By Rayleigh principle, we have that
\begin{align}
    \dfrac{ \mathbf{x}^\intercal H^{(2s,k)} \mathbf{x} }{\mathbf{x}^\intercal \mathbf{x}} \le \xi_1, \qquad
    \dfrac{ \mathbf{w}^\intercal H^{(2s,k)} \mathbf{w} }{\mathbf{w}^\intercal \mathbf{w}} = \xi_2, \label{eq:rayleigh1}
\end{align}

\noindent for every non-zero $\mathbf{x}$ and for $\mathbf{w}$ being the eigenvector corresponding to the second eigenvalue of $H^{(2s,k)}$. Similarly, let $\gamma_1 \ge \gamma_2$ be the eigenvalues of $S^{(2s,k)}$. By Perron-Frobenius, we know that $\gamma_1 \ge 0$. If $\gamma_2 < 0$, then $\det(S^{(2s,k)}) < 0$ and the inequality \eqref{H-S posdef} is trivial. If instead $\gamma_2\ge 0$, then
\begin{align}
    \dfrac{  \mathbf{x}^\intercal S^{(2s,k)} \mathbf{x} }{\mathbf{x}^\intercal \mathbf{x}} \ge  \gamma_2  , \qquad  \dfrac{  \mathbf{v}^\intercal S^{(2s,k)} \mathbf{v}  }{\mathbf{v}^\intercal \mathbf{v}} =   \gamma_1, \label{eq:rayleigh2}
\end{align}

\noindent for any non-zero $\mathbf{x}$ and for $\mathbf{v}$ equal to the leading eigenvector of $S^{(2s,k)}$. We plug vectors $\mathbf{v}$ and $\mathbf{w}$ into \eqref{H-S posdef} to obtain
\begin{align*}
    \rho^k \ge \hspace{3pt} \dfrac{ \mathbf{\mathbf{v}}^{\intercal}S^{(2s,k)} \mathbf{v} }{\mathbf{v}^{\intercal} H^{(2s,k)} \mathbf{v}} \ge \dfrac{ \gamma_1 }{ \xi_1},\\
    \rho^k \ge  \dfrac{ \mathbf{\mathbf{w}}^{\intercal}S^{(2s,k)} \mathbf{w} }{\mathbf{w}^{\intercal} H^{(2s,k)} \mathbf{w}} \ge \dfrac{\gamma_2}{\xi_2},
\end{align*}
\noindent where the last inequalities come from \eqref{eq:rayleigh1} and \eqref{eq:rayleigh2}. Multiplying both inequalities we obtain
\begin{align*}
    \rho^{2k} \ge \dfrac{\gamma_1 \gamma_2 }{\xi_1 \xi_2} = \dfrac{ \det(S^{(2s,k)})}{\det(H^{(2s,k)})}.
\end{align*}
\end{proof}

We are now ready to prove the following corollary.
\begin{corollary} \label{cor:largest-root}
Let $\bm$ be the sequence of moments of a measure supported on the spectrum of $\G$. For $s, k \in \mathbb{N}_0$, let $H^{(2s,k)}$ and $S^{(2s,k)}$ be defined as in \eqref{eq:HS-matrices} and define the following matrix:
\begin{align*}
    F^{(2s,k)}\coloneqq \left[\begin{array}{ll}{m_{2s+k}} & {m_{2s+3k}} \\ {m_{2s}} & {m_{2s+2k}}\end{array}\right].
\end{align*}
Then, whenever $\det\left(H^{(2s,k)}\right) \neq 0$, we have
 \begin{align}
     \rho \ge \left( \dfrac{ \left| \det\left( F^{(2s, k)}\right)\right| + \sqrt{\det\left( F^{(2s, k)}\right)^2 - 4 \det\left(H^{(2s,k)} S^{(2s,k)}\right)}}{2 \det\left(H^{(2s,k)}\right)}\right)^{1/k}.
 \end{align}
\end{corollary}
\begin{proof}
The inequality \eqref{H-S posdef} implies that $\det \left(\rho^k H^{(2s,k)} + S^{(2s,k)}\right) \ge 0$. This can be expanded to
\begin{align*}
    \det\left(\left[\begin{array}{ll}{\rho^k m_{2s} + m_{2s+k}} & {\rho^k m_{2s+k} + m_{2s+2k}} \\ {\rho^k m_{2s+k} + m_{2s+2k}} & {\rho^k m_{2s+2k} + m_{2s+3k}}\end{array}\right]\right) \ge 0,
\end{align*}
\noindent which simplifies to
\begin{align}
    \det\left(H^{(2s,k)}\right)\rho^{2k} - \det\left(F^{(2s,k)}\right) \rho^k + \det\left(S^{(2s,k)}\right) \ge 0. \label{quadratic rho-1}
\end{align}
Similarly, \eqref{H-S posdef} implies that $\det \left(\rho^k H^{(2s,k)} - S^{(2s,k)}\right) \ge 0$, which implies
\begin{align}
    \det\left(H^{(2s,k)}\right)\rho^{2k} + \det\left(F^{(2s,k)}\right) \rho^k + \det\left(S^{(2s,k)}\right) \ge 0. \label{quadratic rho-2}
\end{align}

Inequalities \eqref{quadratic rho-1} and \eqref{quadratic rho-2} are satisfied simultaneously if and only if
\begin{align}
    \det\left(H^{(2s,k)}\right)\rho^{2k} - \left|\det\left(F^{(2s,k)}\right)\right| \rho^k + \det\left(S^{(2s,k)}\right) \ge 0. \label{eq:quadratic-rho}
\end{align}
By Theorem~\ref{thm:hamburger} we have that $\det\left(H^{(2s,k)}\right) > 0$. Using Lemma~\ref{lemma:dets}, we know that $\det\left(S^{(2s,k)}\right) \le \det\left(H^{(2s,k)}\right) \rho^{2k} $, which we substitute into \eqref{eq:quadratic-rho} to yield
\begin{align*}
    2\det\left(H^{(2s,k)}\right)\rho^{2k} - \left| \det\left(F^{(2s,k)}\right)\right| \rho^k \ge 0.
\end{align*}
Since $\rho \ge 0$, we conclude that 
\begin{align}
    \rho^{k} \ge \dfrac{\left| \det\left(F^{(2s,k)}\right)\right|}{2\det\left(H^{(2s,k)}\right)}. \label{eq:larger-than-smallest-root}
\end{align}
Next, define the quadratic polynomial 
\begin{align*}
    P(r) = \det\left(H^{(2s,k)}\right)r^2 - \left|\det\left(F^{(2s,k)}\right)\right| r^k + \det\left(S^{(2s,k)}\right),
\end{align*}
which has a positive leading coefficient. From \eqref{eq:quadratic-rho} we know that $P(\rho^k) \ge 0$ and from \eqref{eq:larger-than-smallest-root} we know that $\rho^k$ is larger than the smallest root of $P$. This implies that $\rho^k$ is larger than the largest root of $P$ and the result follows.
\end{proof}

We can apply Corollary \ref{cor:largest-root} with $s=0$ and $ k=1$ to the closed-walks measure of a graph $\G$, leveraging the fact that $\phi_0 = n$, $\phi_1 = 0$, $\phi_2$ is twice the number of edges, and $\phi_3$ is three times the number of triangles in $\G$, as follows.
\begin{corollary} \label{cor:T-e}
For a graph $\G$ with $n$ vertices, $e$ edges and $T$ triangles, we have that
 \begin{align*}
     \rho \ge \dfrac{3T}{2e} + \sqrt{\left(\dfrac{3T}{2e}\right)^2 + \dfrac{2e}{n}}.
 \end{align*}
\end{corollary}

Similarly, because $\phi_0^{(i)} = 1$, $\phi_1^{(i)} = 0$, $\phi_2^{(i)}$ is the degree of vertex $i$, and $\phi_3^{(i)}$ is twice the number of triangles touching vertex $i$, we can apply Corollary \ref{cor:largest-root} to the closed-walks measure for vertex $i$.
\begin{corollary} \label{cor:ti-di}
 Denoting by $d_i$ the degree of vertex $i$ and $T_i$ the number of triangles touching vertex $i$, we have
 \begin{align*}
     \rho \ge  \max_{i \in 1, \dots, n} \dfrac{T_i + \sqrt{T_i^2 + d_i^3}}{d_i}.
 \end{align*}
\end{corollary}
Notice how Corollary \ref{cor:ti-di} implies
\begin{align*}
    \rho \ge \dfrac{T_{\scriptscriptstyle \Delta} + \sqrt{T_{\scriptscriptstyle \Delta}^2 + \Delta^3}}{\Delta} \ge \sqrt{\Delta},
\end{align*}
where $\Delta$ is the maximum of the vertex degrees in $\G$ and $T_{\scriptscriptstyle \Delta}$ is the maximum triangle count amongst vertices with degree $\Delta$, improving the bound in \eqref{favaron}.\\

\begin{table}[h!]
    \centering
    \begin{adjustbox}{center}
    \def\arraystretch{3.0}
    \begin{tabular}{|c|c|c|c|}
                \hline
        General bound & Special cases & \begin{minipage}{0.78in} \vspace{-3mm} \centering $\vphantom{\sum^k}$ Moment sequence $\vphantom{\sum^k}$ \end{minipage}& Reference\\ \hline
          & $\rho \ge \left(\dfrac{w_{2s+k}}{w_{2s}}\right)^{1/k \hspace{5pt} }$ \cite{nikiforov2006walks} & \begin{minipage}{0.78in}\centering walks  \end{minipage} & \\[1ex] \cline{2-3}
         $\rho \ge \left(\dfrac{m_{2s+k}}{m_{2s}}\right)^{1/k} $ & $\rho \ge \left(\dfrac{\phi_{2s+k}}{\phi_{2s}}\right)^{1/k} $ & closed walks& \begin{minipage}{0.6in} Corollary \ref{basic-niki-bound}  \centering \end{minipage}\\[1ex] \cline{2-3}
         & $\rho \ge \left(\dfrac{\phi_{2s+k}^{(i)}}{\phi_{2s}^{(i)}}\right)^{1/k} $ & \begin{minipage}{0.78in} \vspace{1mm} \centering $\vphantom{\sum^k}$ closed walks from node $i$ \vspace{2mm} \end{minipage}& \\[1ex] \hline
         & $ \rho \ge \dfrac{3T}{2e} + \sqrt{\left(\dfrac{3T}{2e}\right)^2 + \dfrac{2e}{n}}$ & closed walks & \begin{minipage}{0.6in} Corollary \ref{cor:T-e}  \centering \end{minipage} \\[1ex] \cline{2-4}
         $ \rho \ge \left( \dfrac{ \left| \det\left( F^{\scriptscriptstyle(2s, k)}\right)\right| + \sqrt{\det\left( F^{\scriptscriptstyle (2s, k)}\right)^2 - 4 \det\left(H^{\scriptscriptstyle (2s,k)} S^{\scriptscriptstyle (2s,k)}\right)}}{2 \det\left(H^{\scriptscriptstyle (2s,k)}\right)}\right)^{1/k}$ & $\displaystyle \rho \ge  \max_{i \in 1, \dots, n} \dfrac{T_i + \sqrt{T_i^2 + d_i^3}}{d_i}$ & \begin{minipage}{0.78in} \vspace{1mm} \centering $\vphantom{\sum^k}$ closed walks from node $i$ \vspace{2mm} \end{minipage} & \multirow{2}{*}{\begin{minipage}{0.6in} \centering Corollary \ref{cor:ti-di}  \end{minipage} }\\ \cline{2-3}
         & $ \hspace{15pt} \rho \ge \sqrt{\Delta}^{\hspace{15pt}}$ \cite{favaron1993some} & \begin{minipage}{0.78in} \vspace{1mm} \centering $\vphantom{\sum^k}$ closed walks from node $i$ \vspace{2mm} \end{minipage}&  \\ \hline 
    \end{tabular}
    \end{adjustbox}
    \caption{Summary of lower bounds on the spectral radius $\rho$, obtained as corollaries of Lemma~\ref{lem:sdp}. The number of $k$-walks, 
    closed $k$-walks and closed $k$-walks from node $i$, in $\G$, are denoted by $w_k$, $\phi_k$ and $\phi_k^{(i)}$, respectively. We write $n$, $e$ and $T$ to denote the 
    number of nodes, edges and triangles in $\G$. We write $d_i$ and $T_i$ to denote the degree of node $i$ and the 
    number of triangles touching node $i$, respectively. The largest 
    node degree is denoted by $\Delta$.} 
    \label{tab:lower-bounds}
\end{table}

%%%%%%%%%%%%%%%%%%%%%%%%%%%%%%%%%%%%%%%%%%%%%%%%%%%%%
\section{Upper bounds on the spectral radius} \label{sec:upper}

In this section, we make use of Theorems \ref{thm:hamburger} and \ref{thm:stieltjes} to derive upper bounds on $\rho$. These bounds are based on the analysis of three new measures similar to the ones in \eqref{def:spectral-measures}. In particular, these new measures, denoted by $\tilde{\mu}_{\G}$, $\tilde{\mu}_{\G}^{(i)}$, and $\tilde{\nu}_{\G}$, are the result of excluding the summand corresponding to the Dirac delta centered at $\lambda_1 = \rho$ from the definitions of $\mu_{\G}$, $\mu_{\G}^{(i)}$ and $\nu_{\G}$, respectively. Therefore, these three new measures are supported on the set $\{\lambda_2, \lambda_3, \dots, \lambda_n\} \subset [-\rho, \rho]$, and their moments are, respectively,  
\begin{align}
    m_k(\tilde{\mu}_{\G}) = \sum_{l=2}^n \lambda_l^k &= \phi_{k} - \rho^k, \label{moms-bulk-cw}\\
    m_k(\tilde{\mu}_{\G}^{(i)}) = \sum_{l=2}^n c_l^{(i)}\lambda_l^k   &= \phi_{k}(i) - c_1^{(i)}\rho^k, \label{moms-bulk-cw-i}\\
     m_k(\tilde{\nu}_{\G}) = \sum_{l=2}^n c_l \lambda_l^k &= w_k - c_1 \rho^k. \label{moms-bulk-w}
\end{align}
Applying Hamburger's Theorem~to these measures we obtain the following result:
\begin{lemma} \label{lem:upper-bounds-phi}
Let $\bm$ be the sequence of moments of an atomic measure $\mu(x) = \sum_{i=1}^n \alpha_i \delta(x- \lambda_i)$ supported on the spectrum of $\G$, and define the infinite-dimensional Hankel matrix $P$ given by:
\begin{align*}
    P :=\left[\begin{array}{cccc}
    1 & \rho & \rho^2 & {\dots} \\ \rho & \rho^2 & \rho^3 & {\dots} \\
    {\rho^{2}} & {\rho^{3}} & {\rho^{4}} & {\dots} \\
    {\vdots} & {\vdots} & {\vdots} & {\ddots}  \end{array}\right] .
\end{align*}
Hence, for any finite $\mathcal{J} \subset \mathbb{N}_0$,
\begin{align}
H_{\mathcal{J}}(\bm) - \alpha_1 P_{\mathcal{J}} \succeq 0, \label{non-sdp}
\end{align}
\noindent where $H_{\mathcal{J}}(\bm)$ is a submatrix of the Hankel matrix of moments $H_{\max(\mathcal{J})}$ defined in \eqref{def:hankel-moms}.
\end{lemma}
\begin{proof}
We simply note that the matrix $H_{\mathcal{J}}(\mathbf{m}) - \alpha_1 P$ is the Hankel matrix containing the moments of the measure resulting from removing the term corresponding to $\lambda_1$ from the measure $\mu$ supported on the spectrum of $\G$. The result follows directly from Theorem~\ref{thm:hamburger} and Sylvester's criterion.
\end{proof}

%1 by 1 minors and corollaries
\begin{corollary} \label{trivial-bounds}
Let $\bm$ be the sequence of moments of an atomic measure $\mu(x) = \sum_{i=1}^n \alpha_i \delta(x- \lambda_i)$ supported on the spectrum of $\G$. Then  
\begin{align}
    \rho &\le \left(\dfrac{m_{2k}}{\alpha_1}\right)^{1/2k} .
    \end{align}
\end{corollary}
\begin{proof}
  For $\mathcal{J} = \{k+1\}$, Lemma~\ref{lem:upper-bounds-phi} implies 
  \begin{align*}
     m_{2k}- \alpha_1 \rho^{2k} \ge 0.
  \end{align*}
This finishes the proof.
\end{proof}

 Applying this corollary to the measures $\mu_{\G}, \mu_{\G}^{(i)}$ and $\nu_{\G}$, we obtain three different hierarchies of bounds. For example, applying Corollary \ref{trivial-bounds} to the measure $\nu_{\G}$, which has moments $\bm = \{w_s\}_{s=1}^\infty$, we obtain the bound $\rho \le (w_{2k}/c_1)^{1/2k}$, where $c_1 = \left(\sum_{i=1}^n u_{i1}\right)^2$ is the fundamental weight. One can prove that this bound is tighter than the bound \eqref{nikiforov-bound} proved by Nikiforov \cite{nikiforov2006walks} (albeit only for even exponents). In particular, rearranging Wilf's inequality \eqref{wilf}, we obtain
\begin{align}
\left(1 - \frac{1}{\omega(\G)}\right)  &\le  \frac{\rho}{c_1}. \label{wilf-rearranged}
\end{align}
Moreover, the upper bound \eqref{nikiforov-bound} can be expressed as
\begin{align*}
    \left( \left(1 - \frac{1}{\omega(\mathcal{G})}\right) w_{2k}\right)^{\frac{1}{2k + 1}}.
\end{align*}
By substituting \eqref{wilf-rearranged} into this upper bound, we obtain
\begin{align*}
    \left( \left(1 - \frac{1}{\omega(\mathcal{G})}\right) w_{2k}\right)^{\frac{1}{2k + 1}} \ge \left( \rho \frac{w_{2k}}{c_1}\right)^{\frac{1}{2k + 1}} \ge \left( \left(\frac{w_{2k}}{c_1}\right)^{\frac{1}{2k}} \frac{w_{2k}}{c_1}\right)^{\frac{1}{2k + 1}} = \left(\frac{w_{2k}}{c_1}\right)^{\frac{1}{2k}},
\end{align*}
where the last quantity is the upper bound from Corollary \ref{trivial-bounds}, which is less or equal to the bound in \eqref{nikiforov-bound}.\\
%2 by 2 minors and corollaries

Using Lemma~\ref{lem:upper-bounds-phi} with larger principal submatrices, we can improve these upper bounds further. The following upper bound is obtained by analyzing the case of $\mathcal{J} = \{1, k+1\}$.
\begin{corollary} \label{2x2-bounds}
Let $\bm$ be the sequence of moments of an atomic measure $\mu(x) = \sum_{i=1}^n \alpha_i \delta(x- \lambda_i)$ supported on the spectrum of $\G$. Then, for any $k \in \mathbb{N}$
 \begin{align}
    \rho \le \left(\dfrac{m_k + \sqrt{\left(\frac{m_0}{\alpha_1} - 1\right)\left(m_0m_{2k} - m_k^2\right) }}{ \vphantom{\sqrt{\frac{m_0}{\alpha_1}}} m_0}\right)^{1/k} . \label{largest-root}
 \end{align}
Furthermore, this bound is tighter than the one in Corollary \ref{trivial-bounds}.
\end{corollary}

\begin{proof}
Applying Lemma~\ref{lem:upper-bounds-phi} with $\mathcal{J} = \{1, k+1\}$, we conclude that $\det (H_\mathcal{J}(\bm) - \alpha_1 P_\mathcal{J}) \ge 0$, which simplifies to the following expression:
 \begin{align*}
     -m_0\rho^{2k} + 2m_k \rho^k + \frac{1}{\alpha_1}\left((m_0 - \alpha_1 )m_{2k} - m_k^2\right) \ge 0.
 \end{align*}
 Making the substitution $y = \rho^k$, we obtain the following quadratic inequality
 \begin{align*}
     -m_0y^2 + 2m_k y + \frac{1}{\alpha_1}\left((m_0 - \alpha_1)m_{2k} - m_k^2\right) \ge 0.
 \end{align*}
 The quadratic on the left-hand side has a negative leading coefficient, which implies it is 
 negative whenever $y$ is larger than its largest root, which is given by the right hand side of
 \eqref{largest-root}. After substituting back $\rho^k$, the result follows. To see that this 
 bound improves the one in Corollary \ref{trivial-bounds}, note that Corollaries
 \ref{trivial-bounds} and \ref{basic-niki-bound} imply
 \begin{align*}
     m_{2k} \ge \alpha_1 \rho^{2k} \ge \alpha_1 \dfrac{m_{4k}}{m_{2k}} \hspace{8pt} \implies \hspace{8pt} m_{4k} \le \dfrac{1}{\alpha_1} m_{2k}^2 \hspace{8pt} \implies \hspace{8pt} m_0m_{4k} - m_{2k}^2 \le \left(\dfrac{m_0}{\alpha_1} - 1\right)m_{2k}^2.
 \end{align*}
Hence, we can use the inequality in Corollary \ref{largest-root} to obtain
\begin{align*}
    \rho^{2k} \le \dfrac{m_{2k} + \sqrt{\left(\dfrac{m_0}{\alpha_1} - 1\right)\left(m_0m_{4k} - m_{2k}^2\right) }}{m_0} \le \dfrac{m_{2k} + \sqrt{\left(\dfrac{m_0}{\alpha_1} - 1\right)\left(\dfrac{m_0}{\alpha} - 1\right)m_{2k}^2 }}{m_0} = \dfrac{m_{2k}}{\alpha_1}.
\end{align*}
\end{proof}

As with previous results, we can obtain concrete bounds from this result by substituting $\alpha_1$ and $m_k$ by either $(i)$ 1 and $\phi_k$, $(ii)$ $c_1^{(i)}$ and $\phi_k^{(i)}$, or $(iii)$ $c_1$ and $w_k$, respectively. For example, we can apply Corollary \ref{2x2-bounds} to the closed-walks measure for node $i$, $\mu_{\G}^{(i)}$, using $\mathcal{J} = \{1, 2\}$. Since $\phi_0(i) = 1$ and $\phi_1(i) = 0$, we obtain the upper bound in the following Corollary.
\begin{corollary} \label{not-better-phi-j}
    For a graph $\G$
    \begin{align}
        \rho^2 \le \left( \dfrac{1}{c_1^{i}} - 1\right) \phi_2(i) = \left(\dfrac{1}{x_i^2} - 1\right) d_i, \hspace{12pt} \text{for all } i \in \{1, \dots, n\}, \label{bound on xj}
    \end{align}
    where $d_i$ is the degree of vertex $i$ and $x_i$ is the $i$-th component of the leading eigenvector of $A$.
\end{corollary}

Notice that inequality \eqref{bound on xj} can be written as $$ x_i \le \dfrac{1}{\sqrt{1 + \frac{\rho^2}{d_i}}}\hspace{3pt},$$ which was first proven by Cioabă and Gregory in \cite{cioaba2007principal}. Our method provides an alternative proof. Furthermore, we can refine Corollary \ref{not-better-phi-j}, as follows. 
\begin{corollary} \label{maybe-better-bipartite}
    For a bipartite graph $\G$ and $k \in \mathbb{N}_0$, we have
    \begin{align}
        \rho^{2k} &\le \dfrac{\phi_{2k}}{2}, \label{half-bound}\\
        \rho^{2k} &\le \dfrac{\phi_{2k}(i)}{2c_1^{(i)}}. \label{half-bound-i}
    \end{align}
\end{corollary}
\begin{proof}
To prove our result, we define a new measure $\mu_{\G}^+ = \sum_{i=1}^{\lceil n/2 \rceil} 2\delta\left(x-\lambda_{i}^2\right)$. Notice that, since the eigenvalue spectrum of a bipartite graph is symmetric, we have that $m_k\left(\mu_{\G}^+\right) = \sum_{i=1}^{\lceil n/2 \rceil} 2\lambda_i^{2k} = m_k\left(\mu_{\mathcal{G}}\right) = \phi_{2k}$. The upper bound given by \eqref{half-bound} then follows by adapting the proof of Corollary \ref{trivial-bounds} to the measure $\mu_{\G}^+$.

A similar construction can be done for the closed-walks measure $\mu_{\G}^(i)$ from node $i$. It is easy to see that 
\begin{align}
       \phi_{2k}^{(i)} &=  \sum_{j=1}^{\lceil n/2 \rceil} \left(c_j^{(i)} + c_{n-j}^{(i)}\right)\lambda_i^{2k}. \label{half-bound-i-moments}
\end{align}

In what follows, we prove that $c_j^{(i)} = c_{n-j}^{(i)}$ using the \textit{eigenvector-eigenvalue identity} \cite{denton2019eigenvectors}. We first prove that $u_{i,j} = u_{n-i, j}$ for the case of odd $n$, and note that for even $n$ there is an analogous proof. Let $M_{\{j\}}$ be the principal minor of $A$ obtained by deleting row $j$ and column $j$, and let $\gamma_1 \ge \gamma_2 \ge \dots \ge \gamma_{n-1}$ be the eigenvalues of $M_{\{j\}}$. Because $M_{\{j\}}$ is the adjacency matrix of the graph obtained by deleting node $j$, which is also bipartite, its spectrum is also symmetric. From the \textit{eigenvector-eigenvalue identity} we have:
\begin{align}
    u_{i, j}^{2} \prod_{l=1 ; l \neq i}^{n}\left(\lambda_{i} -\lambda_{l} \right) = \prod_{l=1}^{n-1}\left(\lambda_{i} - \gamma_{l}\right). \label{eq:eigenvalue-eigenvector}
\end{align}
Using the symmetry of the spectrum of $M_{\{j\}}$, the term on the right hand side can be rewritten as
\begin{align}
\prod_{l=1}^{n-1}\left(\lambda_{i} - \gamma_{l}\right) = \prod_{l=1}^{(n-1)/2} \left(\lambda_{i}^2 - \gamma_{l}^2\right). \label{right-term-ev-ew}
\end{align}
Similarly, the symmetry of the spectrum of $A$ implies $\lambda_{n-i} = -\lambda_i$ and $\lambda_{\lceil n/2 \rceil} = 0$. Thus, the term accompanying $u_{i,j}^2$ in the left hand side can be rewritten as
\begin{align}
    \prod_{l=1 ; l \neq i}^{n}\left(\lambda_{i} -\lambda_{l} \right) = \left(2 \lambda_i\right) \lambda_i \prod_{\substack{l=1 \\ l \neq i, n-i}} \left(\lambda_i^2 - \lambda_l^2\right), \label{left-term-ev-ew}
\end{align}

\noindent where the first factor on the right corresponds to the term $\left( \lambda_i - \lambda_{n-i}\right)$, the second factor corresponds to $\left(\lambda_i - \lambda_{\lceil n/2 \rceil} \right)$, and the third factor corresponds to the pairs of remaining eigenvalues. After these substitutions, it becomes clear that solving for $u_{ij}^2$ in \eqref{eq:eigenvalue-eigenvector} yields the same result as solving for $u_{n-i,j}^2$. We conclude that $c_j^{(i)} = c_{n-j}^{(i)}$.

Hence, we can conclude from \eqref{half-bound-i-moments} that 
\begin{align*}
    \phi_{2k}^{(i)} = \sum_{l=1}^{n/2} 2c_j^{(i)} \lambda_i^{2k}.
\end{align*}
Thus, if we define the measure
 \begin{align*}
     \mu_{\G}^{(+,j)}(x)&= \sum_{i=1}^{\lceil n/2 \rceil} 2c_i^{(j)}\delta\left(x-\lambda_{i}^2\right),
 \end{align*}
\noindent then $m_k (\mu_{\G}^{(+,j)}) = \phi_{2k}$ and the upper bound \eqref{half-bound-i} follows from adapting the proof of Corollary \ref{trivial-bounds} to this measure.

\end{proof}

A more general version of Corollary \ref{2x2-bounds} is given in the following theorem.
\begin{theorem} \label{conditions for bounds}
Let $\bm$ be the sequence of moments of an atomic measure $\mu(x) = \sum_{i=1}^n \alpha_i \delta(x- \lambda_i)$ supported on the spectrum of $\G$. Define the infinite dimensional Hankel matrix $R$ given by
\begin{align*}
    R :=\left[\begin{array}{cccc}
    1 & r & r^2 & {\dots} \\ r & r^2 & r^3 & {\dots} \\
    {r^{2}} & {r^{3}} & {r^{4}} & {\dots} \\
    {\vdots} & {\vdots} & {\vdots} & {\ddots}  \end{array}\right] .
\end{align*}
Let $\mathcal{J} = \{j_1, j_2, \dots, j_s\}$ and $\mathcal{J}^{\prime} = \{j_1, j_2, \dots, j_{s-1}\}$ for $j_1, \dots, j_s \in \mathbb{N}_0$ such that $H_{\mathcal{J}^\prime}(\bm) \succ 0$. Then, the largest root $r^*$ of the polynomial:
\begin{align*}
   Q(r) \coloneqq  \det \left(H_{\mathcal{J}}(\bm) - \alpha_1 R_J\right),
\end{align*}
is an upper bound on the spectral radius.
\end{theorem}
\begin{proof}
We will prove that $Q(r)$ has a negative leading coefficient equal to $- \alpha_1 \det\left(H(\bm)_{\mathcal{J}'}\right) $. It is well known (see for example \cite{goberstein1988evaluating}) that if $A \in \mathbb{R}^{n\times n}$ and $B := \mathbf{uv}^{\intercal}$ is a rank-1 matrix in $\mathbb{R}^{n \times n}$ then 
 \begin{align} 
     \det(A + B) = \det(A) + \mathbf{v}^{\intercal} \operatorname{adj}(A) \mathbf{u},\label{nice-trick}
 \end{align}
where $\operatorname{adj}(A)$ is the cofactor matrix of $A$. Let $\mathbf{r} \coloneqq \left(r^{j_1 - 1}, \dots, r^{j_s - 1}\right)$. We note that 
\begin{align*}
-\alpha_1 R_\mathcal{J} = \left(\alpha_1\left(r^{j_1 - 1}, \dots, r^{j_s - 1}\right)\right)^{\intercal} \left(-\left(r^{j_1 - 1}, \dots, r^{j_s - 1}\right)\right) = (\alpha_1 \mathbf{r})(-\mathbf{r})^{\intercal}.
\end{align*}
Using \eqref{nice-trick}, we obtain
\begin{align} \label{incredible determinants}
    \det(H(\bm)_\mathcal{J} - \alpha_1 R_\mathcal{J}) = \det(H(\bm)_\mathcal{J}) - \alpha_1 \mathbf{r}^{\intercal} \operatorname{adj}(H(\bm)_\mathcal{J}) \mathbf{r}.
\end{align}
It follows that the leading term of $Q(r)$ is $-\alpha C_{s,s}r^{2(j_{s}-1)} = -\alpha \det(H_{\mathcal{J}'}) < 0$. By Lemma~\ref{lem:upper-bounds-phi} we have $Q(\rho) \ge 0$ and,  therefore, $\rho \le r^*$.
\end{proof}

Lemma~\ref{lem:upper-bounds-phi} was proved applying Hamburger's Theorem~to the moment sequence $\{m_s - \alpha_1 r^s\}_{s=0}^{\infty}$. We can also apply Stieltjes' Theorem~to the same moment sequence to obtain a different hierarchy of upper bounds. This is stated in the following theorem.
\begin{theorem} \label{thm:upper-bounds-stieljtes}
Let $\bm$ be the sequence of moments of an atomic measure $\mu(x) = \sum_{i=1}^n \alpha_i \delta(x- \lambda_i)$ supported on the spectrum of $\G$. Then, for any $\mathcal{J} \in \mathbb{N}_0$,
\begin{align}
\rho \left(H(\bm)_{\mathcal{J}} - \alpha P_{\mathcal{J}}\right) + \left(S(\bm)_{\mathcal{J}}- \alpha_1 \rho P_{\mathcal{J}}\right) \succeq 0, \label{non-sdp-stieltjes}
\end{align}
\noindent where $H(\bm)$ and $S(\bm)$ are the Hankel matrices of moments defined in Theorems \ref{thm:hamburger} and \ref{thm:stieltjes}, respectively. 
\end{theorem}
\begin{proof}
Recall from Lemma~\ref{lem:upper-bounds-phi} that the sequence $\{m_s - \alpha_1 \rho^s \}_{s=0}^\infty$ corresponds to the moment sequence of a measure whose support is contained in $[-\rho, \rho]$ and, therefore, the result follows from Theorem~\ref{thm:stieltjes} . 
\end{proof}

Theorem~\ref{thm:upper-bounds-stieljtes} can be used to obtain bounds that improve on those of Corollary \ref{trivial-bounds}, as shown below.
\begin{corollary} \label{polynomial}
Let $\bm$ be the sequence of moments of an atomic measure $\mu(x) = \sum_{i=1}^n \alpha_i \delta(x- \lambda_i)$ supported on the spectrum of $\G$. Then, the largest root $r^*$ of the following polynomial:
    \begin{align*}
    Q(r) \coloneqq m_{2k} r + m_{2k+1} - 2\alpha_1 r^{2k+1},
    \end{align*}
\noindent is an upper bound on the spectral radius. Furthermore, this bound is tighter than the bound in Corollary \ref{trivial-bounds}.
\end{corollary}
\begin{proof}
 Applying Theorem~\ref{thm:upper-bounds-stieljtes} with $\mathcal{J} = \{k+1\}$, we obtain
\begin{align*}
    \rho \left[m_{2k}\right] +  \left[m_{2k+1}\right] -2 \alpha_1 \rho \left[\rho^{2k}\right] \succeq 0 \implies  \rho m_{2k} +  m_{2k+1} -2 \alpha_1 \rho^{2k+1} \ge 0,
\end{align*}
\noindent and, thus, $Q(\rho) \ge 0$. Since the leading coefficient of $Q(r)$ is negative, it follows that $\rho \le r^*$. To prove that $r^* \le (m_{2k}/\alpha_1)^{1/2k}$, we first prove that $r^*$ is the unique root of $Q(r)$ in the region defined by 
\begin{align}
    r \ge \left(\dfrac{m_{2k}}{2\alpha_1 (2k+1)}\right)^{1/2k}. \label{interval}
\end{align}
This is indeed the case, because the derivative of $Q(r)$, given by $Q^\prime(r) = m_{2k} - 2\hspace{1pt}(2k+1)\hspace{1pt}\alpha_1\hspace{1pt}r^{\scriptscriptstyle2k + 1}$, is negative in the interval defined in \eqref{interval}. Also, note that
\begin{align*}
    \left(\dfrac{m_{2k}}{\alpha_1} \right)^{1/2k} \ge \left(\dfrac{m_{2k}}{2\alpha_1 (2k+1)}\right)^{1/2k}.
\end{align*}
Therefore, it suffices to show that 
\begin{align*}
    Q\left(\left(\dfrac{ m_{2k} }{\alpha_1} \right)^{1/2k}\right) \le 0.
\end{align*}
To this end, we evaluate and obtain

\begin{align*}
    &Q\left(\left(\dfrac{1}{\alpha_1} m_{2k}\right)^{1/2k}\right) \le 0\\
    &\iff \left(\dfrac{1}{\alpha_1} m_{2k}\right)^{1/2k} m_{2k} +  m_{2k+1} -2 \alpha_1 \left(\left(\dfrac{1}{\alpha_1} m_{2k}\right)^{1/2k}\right)^{2k+1}\le 0\\
    &\iff \left(\dfrac{1}{\alpha_1} m_{2k}\right)^{1/2k} m_{2k} +  m_{2k+1} - 2 m_{2k} \left(\dfrac{1}{\alpha_1} m_{2k}\right)^{1/2k} \le 0\\
    &\iff \dfrac{m_{2k+1}}{m_{2k}} \le \left(\dfrac{1}{\alpha_1} m_{2k}\right)^{1/2k},
\end{align*}
where the last inequality is true since the left-hand side is a lower bound of $\rho$ by Corollary \ref{basic-niki-bound}, and the right hand side is an upper bound of $\rho$ by Corollary \ref{trivial-bounds}. This finishes the proof.  
\end{proof}

The implicit bound in Corollary \ref{polynomial}, when applied to the moment sequence $\bm = \{w_s\}_{s=1}^\infty$, provides an improvement on the bound given in Corollary \ref{basic-niki-bound} and, consequently, on the bound \eqref{niki-bound}. Using inequality \eqref{wilf}, we can also obtain a bound in terms of the clique number instead of the fundamental weight, which also improves on \eqref{niki-bound}, as we show below.
\begin{corollary} \label{polynomial-clique}
The largest root $r^*$ of the following polynomial:
\begin{align*}
    Q(r) \coloneqq m_{2k}r  + m_{2k+1} - 2\dfrac{\omega(\G)}{\omega(\G) - 1} r^{2k+2},
\end{align*}
\noindent is an upper bound on the spectral radius. Furthermore, this bound is an improvement on \eqref{nikiforov-bound}.
\end{corollary}
\begin{proof}
 The proof is very similar to that of Corollary \ref{polynomial}. Using \eqref{wilf}, we have that
 \begin{align*}
         \rho w_{2k} + w_{2k+1} - \dfrac{\omega(\G)}{\omega(\G) - 1}\rho^{2k+1} \ge \rho w_{2k} + w_{2k+1} - 2\alpha_1 \rho^{2k+1} \ge 0.
 \end{align*}
We omit the details to avoid repetitions. 
\end{proof}

\begin{table}[h!]
    \centering
    \begin{adjustbox}{center}
        \def\arraystretch{3.0}
    \begin{tabular}{|c|c|c|c|}
                \hline
        General bound & Special cases & \begin{minipage}{0.78in} \vspace{-3mm} \centering $\vphantom{\sum^k}$ Moment sequence $\vphantom{\sum^k}$ \end{minipage}& Reference\\ \hline
          & 
            $\rho \le \left(\phi_{2k}\right)^{1/2k}$
          & \begin{minipage}{0.78in}\centering closed walks \end{minipage} & \\[1ex] \cline{2-3}
            &$\rho \le \left(\dfrac{w_{2k}}{c_1}\right)^{1/2k}$ & all & \begin{minipage}{0.65in}\centering Corollary \ref{trivial-bounds} \end{minipage} \\[1ex] \cline{2-3}
         $\rho \le \left(\dfrac{m_{2k}}{\alpha_1}\right)^{1/2k}$&  $\qquad \rho \le \left(\left(1 - \dfrac{1}{\omega(\G)}\right)w_{2k}\right)^{1/2k+1}$\hspace{3pt} \cite{nikiforov2006walks} & \begin{minipage}{0.78in} \vspace{1mm} \centering $\vphantom{\sum^k}$ walks \vspace{2mm} \end{minipage}& \\[1ex] \cline{2-4}
         & $\rho \le \left(\dfrac{\phi_{2k}}{2}\right)^{1/2k}$& \begin{minipage}{0.78in} \vspace{1mm} \centering $\vphantom{\sum^k}$ closed walks (bipartite $\G$)\vspace{2mm} \end{minipage} & \multirow{2}{*}{\begin{minipage}{0.65in}\centering  Corollary \ref{maybe-better-bipartite} \vspace{-14pt} \end{minipage}}\\
         \cline{2-3}
         & $\rho \le \left(\dfrac{\phi_{2k}^{(i)}}{2}\right)^{1/2k}$& \begin{minipage}{0.78in} \vspace{1mm} \centering $\vphantom{\sum^k}$ closed walks from node $i$ (bipartite $\G$) \vspace{-1mm} \end{minipage}
         &\\[1ex] \hline 
         $\rho \le \left(\dfrac{m_k + \sqrt{\left(\dfrac{m_0}{\alpha_1} - 1\right)\left(m_0m_{2k} - m_k^2\right) }}{m_0}\right)^{1/k} $ & \hspace{42pt} $\rho \le \sqrt{\left(\dfrac{1}{x_i^2}-1\right)d_i}$\hspace{3pt} \cite{cioaba2007principal} & \begin{minipage}{0.78in} \vspace{6mm} \centering $\vphantom{\sum^k}$ closed walks from node $i$ \vspace{6mm} \end{minipage} & \begin{minipage}{0.65in}\centering Corollary \ref{not-better-phi-j} \end{minipage} \\[2.5ex] \hline
         \multirow{2}{*}{         $\begin{array}{ccl}
                \multirow{2}{*}{$\rho \le \vspace{10pt}$ } & {  \max\limits_{r}} & {r} \\[-5ex]
                & {\text{s.t}} & { m_{2k}r + m_{2k+1} - 2\alpha_1 r^{2k+1} = 0}
            \end{array}$ \vspace{-6.5mm}} & $\begin{array}{ccl}
                \multirow{2}{*}{$\rho \le \vspace{10pt}$ } & {  \max\limits_{r}} & {r} \\[-5ex]
                & {\text{s.t}} & { w_{2k}r + w_{2k+1} - 2c_1 r^{2k+1} = 0}
            \end{array}$ & \begin{minipage}{0.78in} \vspace{1mm} \centering $\vphantom{\sum^k}$ walks \vspace{2mm} \end{minipage} & \begin{minipage}{0.65in}\centering  Corollary \ref{polynomial} \end{minipage}\\[1ex] \cline{2-4}
 &          $\begin{array}{ccl}
                \multirow{2}{*}{$\rho \le \vspace{10pt}$ } & {  \max\limits_{r}} & {r} \\[-5ex]
                & {\text{s.t}} & { w_{2k}r + w_{2k+1} - 2\dfrac{\omega(\G)}{\omega(\G) - 1} r^{2k+2} = 0}
            \end{array}$ & \begin{minipage}{0.78in} \vspace{1mm} \centering $\vphantom{\sum^k}$ walks \vspace{1mm} \end{minipage} & \begin{minipage}{0.65in}\centering  Corollary \ref{polynomial-clique} \end{minipage}\\[3ex]\hline
    \end{tabular}
    \end{adjustbox}
    \caption{Upper bounds on the spectral radius $\rho$, where $ m_k $ is the $k$-th moment of an atomic measure $\mu(x) = \sum_{i=1}^n \alpha_i \delta(x- \lambda_i)$ supported on the spectrum of $\G$. The number of $k$-walks,  closed $k$-walks and closed $k$-walks from node $i$ are denoted by $w_k$, $\phi_k$ and $\phi_k^{(i)}$, respectively. We write $\omega(\G)$, $c_1$ and $x_i$ for the clique number of $\G$, the fundamental weight of $A$, and the $i$-th component of the leading eigenvector of $A$, respectively.}
    \label{tab:upper-bounds}
\end{table}

%\textbf{Acknowledgement } We are grateful to $\_\_\_\_\_\_\_\_\_$ for useful comments that helped improve the paper. 
\FloatBarrier
\bibliographystyle{plain}
\bibliography{biblio}

\begin{thebibliography}{10}

\bibitem{boyd2004convex}
Stephen Boyd and Lieven Vandenberghe.
\newblock {\em Convex optimization}.
\newblock Cambridge university press, 2004.

\bibitem{cioabua2007large}
Sebastian~M Cioab{\u{a}} and David~A Gregory.
\newblock Large matchings from eigenvalues.
\newblock {\em Linear Algebra and its Applications}, 422(1):308--317, 2007.

\bibitem{cioaba2007principal}
Sebastian~M Cioaba and David~A Gregory.
\newblock Principal eigenvectors of irregular graphs.
\newblock {\em Electronic Journal of Linear Algebra}, 16(1):31, 2007.

\bibitem{cvetkovic1972chromatic}
D~Cvetkovic.
\newblock Chromatic number and the spectrum of a graph.
\newblock {\em Publ. Inst. Math.(Beograd)}, 14(28):25--38, 1972.

\bibitem{denton2019eigenvectors}
Peter~B Denton, Stephen~J Parke, Terence Tao, and Xining Zhang.
\newblock Eigenvectors from eigenvalues.
\newblock {\em arXiv preprint arXiv:1908.03795}, 2019.

\bibitem{edwards1983lower}
CS~Edwards and CH~Elphick.
\newblock Lower bounds for the clique and the chromatic numbers of a graph.
\newblock {\em Discrete Applied Mathematics}, 5(1):51--64, 1983.

\bibitem{favaron1993some}
Odile Favaron, Maryvonne Mah{\'e}o, and J-F Sacl{\'e}.
\newblock Some eigenvalue properties in graphs (conjectures of graffiti—ii).
\newblock {\em Discrete Mathematics}, 111(1-3):197--220, 1993.

\bibitem{goberstein1988evaluating}
Simon~M Goberstein.
\newblock Evaluating “uniformly filled” determinants.
\newblock {\em The College Mathematics Journal}, 19(4):343--345, 1988.

\bibitem{hofmeister1988spectral}
Michael Hofmeister.
\newblock Spectral radius and degree sequence.
\newblock {\em Mathematische Nachrichten}, 139(1):37--44, 1988.

\bibitem{hong2005sharp}
Yuan Hong and Xiao-Dong Zhang.
\newblock Sharp upper and lower bounds for largest eigenvalue of the laplacian
  matrices of trees.
\newblock {\em Discrete Mathematics}, 296(2-3):187--197, 2005.

\bibitem{horn2012matrix}
Roger~A Horn and Charles~R Johnson.
\newblock {\em Matrix analysis}.
\newblock Cambridge university press, 2012.

\bibitem{meyer2000matrix}
Carl~D Meyer.
\newblock {\em Matrix analysis and applied linear algebra}, volume~71.
\newblock SIAM, 2000.

\bibitem{nikiforov2006walks}
Vladimir Nikiforov.
\newblock Walks and the spectral radius of graphs.
\newblock {\em Linear Algebra and its Applications}, 418(1):257--268, 2006.

\bibitem{nikiforov2007bounds}
Vladimir Nikiforov.
\newblock Bounds on graph eigenvalues ii.
\newblock {\em Linear Algebra and its Applications}, 427(2-3):183--189, 2007.

\bibitem{biggs1993algebraic}
Biggs Norman.
\newblock Algebraic graph theory.
\newblock {\em Cambridge Mathematical Library}, 1993.

\bibitem{preciado2013moment}
Victor~M Preciado and Ali Jadbabaie.
\newblock Moment-based spectral analysis of large-scale networks using local
  structural information.
\newblock {\em IEEE/ACM Transactions on Networking (TON)}, 21(2):373--382,
  2013.

\bibitem{schmudgen2017moment}
Konrad Schm{\"u}dgen.
\newblock {\em The moment problem}, volume~9.
\newblock Springer, 2017.

\bibitem{stevanovic2008spectral}
Dragan Stevanovi{\'c}, Mustapha Aouchiche, and Pierre Hansen.
\newblock On the spectral radius of graphs with a given domination number.
\newblock {\em Linear Algebra and Its Applications}, 428(8-9):1854--1864, 2008.

\bibitem{van2014graph}
Piet Van~Mieghem.
\newblock Graph eigenvectors, fundamental weights and centrality metrics for
  nodes in networks.
\newblock {\em arXiv preprint arXiv:1401.4580}, 2014.

\bibitem{collatz1957spektren}
Lothar Von~Collatz and Ulrich Sinogowitz.
\newblock Spektren endlicher grafen.
\newblock In {\em Abhandlungen aus dem Mathematischen Seminar der
  Universit{\"a}t Hamburg}, volume~21, pages 63--77. Springer, 1957.

\bibitem{west1996introduction}
Douglas~Brent West et~al.
\newblock {\em Introduction to graph theory}, volume~2.
\newblock Prentice hall Upper Saddle River, NJ, 1996.

\bibitem{wilf1986spectral}
Herbert~S Wilf.
\newblock Spectral bounds for the clique and independence numbers of graphs.
\newblock {\em Journal of Combinatorial Theory, Series B}, 40(1):113--117,
  1986.

\bibitem{yu2004spectral}
Aimei Yu, Mei Lu, and Feng Tian.
\newblock On the spectral radius of graphs.
\newblock {\em Linear algebra and its applications}, 387:41--49, 2004.

\end{thebibliography}
\end{document}